\documentclass{amsart}
\usepackage{graphicx}
\usepackage{amssymb, amsmath}
\usepackage{graphicx}
\usepackage{tikz}
\usepackage{cleveref}
\usepackage{bm}
\usepackage{xfrac}

\newtheorem{theorem}{Theorem}[section]
\newtheorem{lemma}[theorem]{Lemma}
\newtheorem{proposition}[theorem]{Proposition}

\newtheorem*{thmnowhere}{Theorem~\ref{thm:nowhere}} 
\newtheorem*{thmmain}{Theorem~\ref{thm:main}} 

\theoremstyle{definition}
\newtheorem{definition}[theorem]{Definition}

\newtheorem{example}[theorem]{Example}
\newtheorem{question}[theorem]{Question}
\newtheorem{conjecture}[theorem]{Conjecture}

\newtheorem*{formulationA}{Formulation A}
\newtheorem*{formulationB}{Formulation B}

\theoremstyle{remark}
\newtheorem{remark}[theorem]{Remark}

\numberwithin{equation}{section}

\newcommand{\demph}[1]{\emph{#1}}

\newcommand{\Met}{\operatorname{Met}}
\newcommand{\FMet}{\operatorname{FMet}}
\newcommand{\FMetmag}{\operatorname{FMet}^\circ}
\newcommand{\sumofentry}{\operatorname{\mathbf{sum}}}
\newcommand{\adj}{\operatorname{adj}}
\newcommand{\diam}{\operatorname{diam}}
\newcommand{\surj}{\twoheadrightarrow}
\newcommand{\lines}[2]{\mathrm{Lines}_{#1}({#2})}

\newcommand{\runij}{\diamondsuit}
\newcommand{\runab}{\bm{\ast}}

\renewcommand{\th}{\text{th}}

\title{Is magnitude `generically continuous' for finite metric spaces?}

\author{Hirokazu Katsumasa}
\address{Department of Mathematics, Osaka University, Japan}
\email{u895471k@ecs.osaka-u.ac.jp}

\author{Emily Roff}
\address{School of Mathematics, University of Edinburgh, Scotland}
\email{emily.roff@ed.ac.uk}

\author{Masahiko Yoshinaga}
\address{Department of Mathematics, Osaka University, Japan}
\email{yoshinaga@math.sci.osaka-u.ac.jp}

\begin{document}

\begin{abstract}
Magnitude is a real-valued invariant of metric spaces which, in the finite setting, can be understood as recording the `effective number of points' in a space as the scale of the metric varies. Motivated by applications in topological data analysis, this paper investigates the \emph{stability} of magnitude: its continuity properties with respect to the Gromov--Hausdorff topology. We show that magnitude is nowhere continuous on the Gromov--Hausdorff space of finite metric spaces. Yet, we find evidence to suggest that it may be `generically continuous', in the sense that generic Gromov--Hausdorff limits are preserved by magnitude. We make the case that, in fact, `generic stability' is what matters for applicability.
\end{abstract}

\maketitle

\setcounter{tocdepth}{1}
% \tableofcontents

%%%%%%%%%%%%%%%%%%%%%%%%%%%%%%%%%%%%%%%%%

\section{Introduction}

Magnitude is a real-valued isometric invariant of metric spaces, with wide-ranging connections to topics in geometry, analysis, combinatorics and category theory. It was first introduced by Leinster in \cite{LeinsterMagnitude2013}, where it was defined for finite metric spaces as follows. Given a space \(X\), let \(Z_X\) denote the \(X \times X\) matrix with entries $Z_X(x,y)=e^{-d(x, y)}$; this is called the \demph{similarity matrix} of \(X\). A \demph{weighting} for \(X\) is a vector \(\bm{w} \in \mathbb{R}^X\) with the property that \((Z_X \bm{w})(x) = 1\) for every \(x \in X\). If a weighting \(\bm{w}\) exists, then the \demph{magnitude} \(|X|\) is defined to be the sum of the entries in \(\bm{w}\). (In the case that \(X\) carries more than one weighting, this value is independent of the choice of weighting used to compute it \cite[Lem.~1.1.2]{LeinsterMagnitude2013}.)

For almost every finite metric space, a weighting does exist---indeed, in the generic case \(\det(Z_X) \neq 0\), so there is a unique weighting on \(X\) given by \(\bm{w}(x) = \sum_{y \in X} Z_X^{-1}(x,y)\). Thus, for a generic finite metric space \(X\), we have
\begin{equation}
\label{eq:mag_def}
    |X| = \sum_{x,y \in X} Z_{X}^{-1}(x,y).
\end{equation}

The magnitude of a metric space is analogous, in a precise sense, to the Euler characteristic of a topological space \cite{LeinsterMagnitude2008}. Unlike Euler characteristic, however, the magnitude of a space depends on the scale of the metric. Thus, given a metric space \(X = (X,d)\), it is informative to consider the magnitudes of the scaled spaces \(tX := (X, t d)\) for each real number \(t > 0\). The \demph{magnitude function} of \(X\) is the partially defined function 
\begin{align}
\label{eq:mag_R}
\begin{split}
    (0, \infty) &\longrightarrow \mathbb{R} \\
    t &\longmapsto |tX|.
\end{split}
\end{align}
Under favourable conditions---for instance, if \(X\) is a subset of Euclidean space---this function is in fact continuous, and satisfies \(\lim_{t \to 0} |tX| = 1\) and \(\lim_{t \to \infty} |tX| = \# X\). In such cases, the magnitude function can be interpreted as recording the `effective number of points' in \(X\) as the scale of the metric varies, and its instantaneous rate of growth as recording the `effective dimension'. For illustrative examples, see \cite[Ex.~6.4.6]{LeinsterEntropy2021} and \cite[\S 4]{WillertonSpread2015}.

The definition of magnitude can be extended from finite metric spaces to a large class of compact metric spaces \cite[\S 2]{MeckesPositive2013}, and in that setting it turns out to be remarkably rich in geometric content. The large-scale asymptotics of the magnitude function are of particular interest when interpreted for domains in Euclidean space \cite{BarceloCarbery2018, GimperleinGoffeng, Willerton2020Odd, Meckes2020} or manifolds with boundary \cite{GimperleinGoffengLouca2022, GimperleinGoffengLouca2024}, capturing information about `size-related' attributes such as volume, surface area, and integrals of curvatures. There is also an associated homology theory, \demph{magnitude homology}, which has been related in valuable ways to graph-theoretic invariants arising in the context of discrete homotopy theory \cite{HepworthWillerton,Asao, HepworthRoff2024,kaneta2021magnitude}.

Though its geometric features are displayed most vividly in the world of compact metric spaces, there has been a recent resurgence of interest in the magnitude of finite spaces, both from the perspective of `pure' metric geometry \cite{OHara2401, OHara2408, Gomi2023} and with a view to applications. Magnitude is increasingly being put to work as an invariant of data sets, in the context of topological data analysis (e.g.~\cite{ABOR2019, OKO2023}) and in machine learning (e.g.~\cite{ALRS2023, Huntsman2023, LASR2023A}), with new techniques being developed to facilitate its efficient computation or approximation \cite{AWSGS}.

Applications such as these bring new theoretical questions into the foreground. In particular, for applicability in data analysis, the \emph{stability} of geometric or topological invariants is important: small perturbations of the input data should lead to small changes, at most, in the value of an invariant \cite[Chap.~VIII]{EdelHarer}. In other words, a good invariant should depend \emph{continuously} on the input data. Indeed, within the topological data analysis literature, proving a `stability theorem' for a given invariant is usually seen as a prerequisite for its adoption in practice. A stability theorem typically says that the invariant in question is Lipschitz continuous, or at least continuous, with respect to the Hausdorff or Gromov--Hausdorff metric on the relevant space of data (e.g.~\cite{CEH2005, CarlssonMemoli2010, QGFHM2024}).

Despite its increasingly widespread application, there is, to date, no stability theorem for magnitude. To describe what such a theorem should look like, let $\FMet$ denote the set of isometry classes of finite metric spaces, equipped with the Gromov--Hausdorff metric, and let \(\FMetmag \subset \FMet\) be the subspace consisting of finite metric spaces for which magnitude is defined. A conventional stability theorem for magnitude might tell us that the mapping
\begin{align}
\label{eq:mag_Fmet}
\begin{split}
    \FMetmag &\longrightarrow \mathbb{R} \\
    X &\longmapsto |X|
\end{split}
\end{align}
is continuous (or, better, Lipschitz continuous). Unfortunately, magnitude is \emph{not} continuous in this sense: examples of discontinuities can be found in \cite[Ex.~2.2.7 and Ex.~2.2.8]{LeinsterMagnitude2013} and \cite[Ex.~3.4]{roff2023small}. Yet, such examples seem to be rare---so rare, in fact, that intuition suggests one is vanishingly unlikely to encounter instances of magnitude's instability in `the wild' of real-world data.

In this paper we take the point of view that, for the purposes of applicability, the \emph{generic} behaviour of invariants is what one needs to understand, not their vanishingly rare pathologies. This seems to reflect actual practice in the domains where magnitude is being applied: for instance, in \cite[\S 3.2]{LASR2023A}, the assumption that \(\lim_{t \to 0} |tX| = 1\) is justified by the fact, proved in \cite[Thm.~2.3]{roff2023small}, that this property holds generically---though not always---for finite metric spaces.

We therefore pose the question: is magnitude \emph{generically stable} as an invariant of finite metric spaces? In other words, is the function \(X \mapsto |X|\) \emph{generically continuous} on \(\FMetmag\)? Of course, the answer will depend on how one chooses to formulate the notion of `generic continuity'. For one natural formulation, we will prove definitively that magnitude is \emph{not} generically continuous. For another natural formulation, we will provide strong evidence to suggest that it may be.

\subsection*{Summary of results}

In the first formulation, a function is generically continuous if it is continuous at a generic point of its domain:

\begin{formulationA}
    Let \(A\) and \(B\) be topological spaces. A function \(f \colon A \to B\) is generically continuous if there exists a dense open subset \(U \subseteq A\) such that \(f\) is continuous at every point of \(U\). Equivalently, \(f\) is generically continuous if there exists a nowhere-dense subset \(V \subset A\) such that \(f|_{A \backslash V} \colon A \backslash V \to B\) is continuous.
\end{formulationA}

In \Cref{sec:discontinuity} we prove that magnitude is not generically continuous in this sense. In fact, to demonstrate that magnitude fails to satisfy Formulation A, we prove the following much stronger statement.

\begin{thmnowhere}
    The function \(X \mapsto |X|\) is nowhere continuous on \(\FMetmag\).
\end{thmnowhere}

\Cref{sec:discontinuity} also contains new examples illustrating just how `badly' continuity can fail. In \cite[Thm.~3.8]{roff2023small} it was shown that for every real number \(\alpha \geq 1\), there exists a finite metric space \(X\) such that \(\lim_{t \to 0} |tX| = \alpha\), but the question was left open of whether there exists \(X\) such that \(\lim_{t \to 0} |tX| < 1\) or \(\lim_{t \to 0} |tX| = +\infty\). That question is not answered here, but in \Cref{ex:perturb} we show that if one considers arbitrary paths \((X_s)_{0 < s \ll 1}\) in \(\FMetmag\) converging to the one-point space, then \(\lim_{s \to 0} |X_s|\) can take any value in the extended real line. We also consider certain paths converging to a two-point space \(K_2\), and describe the associated sequences of magnitude functions. We demonstrate that the pointwise limit of these functions need not display any of the good qualitative features of the magnitude function of \(K_2\), such as continuity or monotonicity with respect to the scale factor \(t\), or meaningful behaviour in the large- or small-scale limit.

Our more positive results are contained in Sections \ref{sec:lines} and \ref{sec:conti}. The idea is to reformulate generic continuity in sequential terms, as follows.
\begin{formulationB}
    Let \(A\) and \(B\) be metric spaces. A function \(f \colon A \to B\) is generically continuous if, for a generic convergent sequence \((X_k)_{k \in \mathbb{N}}\) in \(A\), we have 
    \[\lim_{k \to \infty} f(X_k) = f\left(\lim_{k \to \infty} X_k \right).\]
\end{formulationB}

Formulation B is not quite precise, since is not clear \textit{a priori} what `generic' ought to mean for sequences. Moreover, our techniques do not yet allow us to study the behaviour of magnitude with respect to arbitrary sequences in \(\FMetmag\). Instead, we restrict attention to a particular class of sequences, or paths, which generalize those of the form \((sX)_{0 < s \ll 1}\) and can be thought of as the straight lines in Gromov--Hausdorff space. In \Cref{sec:lines} we define, for each finite metric space \(X\), a topological space of lines converging to \(X\). \Cref{sec:conti} contains our main result, which is a first attempt at a `generic stability theorem' for magnitude:

\begin{thmmain}
    For any finite metric space \(X\) such that \(Z_X\) is invertible, the space of lines to \(X\) contains a dense open subset whose elements satisfy \(\lim_{s \to 0} |Y_s| = |X|\).
\end{thmmain}

Since \(Z_X\) is generically invertible, \Cref{thm:main} can be interpreted as saying that magnitude behaves continuously for generic limits taken along lines in Gromov--Hausdorff space. This offers strong evidence to suggest that magnitude may indeed be generically continuous in the sense of Formulation B.

\subsection*{Open questions}

This paper does not answer the question in the title; we do not even insist upon a particular interpretation of `generic continuity'. Thus, several questions are left open. The first is already apparent:

\begin{question}
\label{q:all_paths}
Is magnitude generically continuous in the sense of Formulation B? That is, does a version of \Cref{thm:main} hold with respect to a space of all sequences in \(\FMetmag\)?
\end{question}

There may also be other reasonable formulations of generic continuity that we have not considered here.

An alternative route to a stability theorem for magnitude would be to look for a class of metric spaces such that if \(X\) belongs to this class, then magnitude is continuous with respect to the Hausdorff topology on finite subsets of \(X\). One good candidate might be the class of \demph{positive definite} spaces. In \cite{MeckesPositive2013}, Meckes defines a positive definite metric space to be one such that every finite subspace has positive definite similarity matrix, and he proves that magnitude is lower semicontinuous---though not continuous---on the Gromov--Hausdorff space of compact positive definite metric spaces \cite[Thm.~2.6]{MeckesPositive2013}. He also poses the following question:

\begin{question}
\label{q:pos_def}
    Given a positive definite metric space \(W\), is magnitude continuous with respect to the Hausdorff topology on the set of compact subspaces of \(W\)?
\end{question}

Since the metric induced by the \(\ell^p\)-norm on \(\mathbb{R}^n\) is positive definite for \(1 \leq p \leq 2\) \cite[Thm.~3.6]{MeckesPositive2013}, a positive answer to \Cref{q:pos_def} would imply in particular that magnitude is continuous for finite subsets of these spaces. This would be especially valuable to know for the purposes of applications, which often concern data embedded into \(\mathbb{R}^n\) with the Euclidean (\(\ell^2\)) or taxicab (\(\ell^1\)) metric.

Moreover, in the setting of finite-dimensional, positive definite normed spaces such as these, the techniques of Meckes \cite{MeckesMagnitude2019} are available, which seem likely to facilitate stronger results on continuity. Indeed, the strongest known result in this direction says that magnitude is continuous on the class of \(n\)-dimensional compact convex sets in any \(n\)-dimensional positive definite normed space \cite[Thm.~4.15]{LeinsterMagnitude2017b}. (In fact, the statement is a bit stronger than that.) And in Leinster and Meckes \cite[Thm.~3.1]{LeinsterMeckes2021}, the techniques of \cite{MeckesMagnitude2019} are used to prove that if \(U\) is a finite-dimensional, positive definite normed space then every compact subset \(X\) of \(U\) satisfies \(\lim_{t \to 0} |tX| = 1\). (Other relevant remarks by Meckes concerning the behaviour of magnitude in this setting can be found in \cite{MeckesMagnitude2019}, following the proof of Theorem 1.) We therefore close this introduction by stating a conjecture.

\begin{conjecture}
\label{conj:euclidean}
Let \(U\) be a finite-dimensional, positive definite normed space. Then the function \(X \mapsto |X|\) is continuous with respect to the Hausdorff topology on the set of finite subsets of \(U\). In particular, magnitude is stable for finite subsets of Euclidean or taxicab space.
\end{conjecture}

\subsection*{Acknowledgements}

This work was partially supported by JSPS KAKENHI JP22K18668, JP23H00081. We are grateful to Sergei O.~Ivanov for useful conversations, and to Mark Meckes for valuable comments on the first version of this paper.

%%%%%%%%%%%%%%%%%%%%%%%%%%%%%%%%%%%%%%%%%

\section{Magnitude is nowhere continuous on Gromov--Hausdorff space}
\label{sec:discontinuity}

In this section, we establish that magnitude fails to be generically continuous in the sense of Formulation A. Indeed, we prove that magnitude is \emph{nowhere} continuous on \(\FMet\), and we construct new examples illustrating the nature of its discontinuities. Our examples all make use of two basic constructions on metric spaces---the wedge and the join---and a lemma, due to Leinster, which describes the behaviour of magnitude with respect to these constructions.

\begin{definition}
    Let $A$ and $B$ be metric spaces with chosen points $a\in A$ and 
    $b\in B$. The \demph{wedge} of $A$ and $B$ over $a$ and $b$, denoted $A\vee_{a,b} B$, 
    is the metric space with point set $A\sqcup B/a \sim b$ and distance function
    \[
    d(u, v)=
    \begin{cases}
        d_A(u, v) & \text{ if } u, v\in A\\
        d_B(u, v) & \text{ if } u, v\in B\\
        d_A(u, a)+d_B(b, v) & \text{ if } u\in A, v\in B. 
    \end{cases}
    \]
    Where the choice of points $a$ and $b$ is understood, or is not important, 
    we will write $A\vee B$ for $A\vee_{a,b} B$
\end{definition}

\begin{definition}
Let $A$ and $B$ be metric spaces. For each \(\ell \geq \frac{1}{2} \diam (A), \frac{1}{2}\diam (B)\), the \demph{distance-\(\ell\) join} of \(A\) and \(B\), denoted by $A+_\ell B$, is the metric space with point set $A\sqcup B$ and distance function 
\[
d(u, v)=
\begin{cases}
d_{A}(u, v)& \text{ if } u, v\in A\\
d_{B}(u, v)& \text{ if } u, v\in B\\
\ell& \text{ if } u\in A, v\in B. 
\end{cases}
\]
Note that $t(A+_\ell B)=tA+_{t\ell}tB$.
\end{definition}

\begin{figure}[htbp]
\centering
\begin{tikzpicture}
%\draw [help lines] (0,0) grid (10,4);

%wedge
\coordinate (A11) at (0.5,1.3); 
\coordinate (A12) at (0.8,0.6); 
\coordinate (A13) at (1.35,1.1); 
\coordinate (AB1) at (2,1); 

\coordinate (B11) at (2.5,0.75); 
\coordinate (B12) at (3,1.2); 
\coordinate (B13) at (3.35,0.7); 
\coordinate (B14) at (3.6,1.1); 

\filldraw[fill=black, draw=black] (A11) circle[radius=0.05];
\filldraw[fill=black, draw=black] (A12) circle[radius=0.05];
\filldraw[fill=black, draw=black] (A13) circle[radius=0.05];
\filldraw[fill=black, draw=black] (AB1) circle[radius=0.05];
\draw[densely dashed] (1,1) circle [x radius=1,y radius=0.6];
\draw (1,1.6) node[above] {$A$};

\filldraw[fill=black, draw=black] (AB1) node[left] {$a$} node[right] {$b$};

\filldraw[fill=black, draw=black] (B11) circle[radius=0.05];
\filldraw[fill=black, draw=black] (B12) circle[radius=0.05];
\filldraw[fill=black, draw=black] (B13) circle[radius=0.05];
\filldraw[fill=black, draw=black] (B14) circle[radius=0.05];
\draw[densely dashed] (3,1) circle [x radius=1,y radius=0.6];
\draw (3,1.6) node[above] {$B$};

%join

\coordinate (A21) at (5.5,-0.2); 
\coordinate (A22) at (6.5,0.1); 
\coordinate (A23) at (7.5,-0.1); 

\filldraw[fill=black, draw=black] (A21) circle[radius=0.05];
\filldraw[fill=black, draw=black] (A22) circle[radius=0.05];
\filldraw[fill=black, draw=black] (A23) circle[radius=0.05];
\draw[densely dashed] (6.5,0) circle [x radius=2,y radius=0.4];
\draw (8.5,0) node[right] {$A$};

\coordinate (B21) at (5.8,2.4); 
\coordinate (B22) at (7.3,2.6); 

\filldraw[fill=black, draw=black] (B21) circle[radius=0.05];
\filldraw[fill=black, draw=black] (B22) circle[radius=0.05];
\draw[densely dashed] (6.5,2.5) circle [x radius=2,y radius=0.4];
\draw (8.5,2.5) node[right] {$B$};

\draw[ultra thin] (A21)--(B21);
\draw[ultra thin] (A22)--(B21);
\draw[ultra thin] (A23)--(B21);
\draw[ultra thin] (A21)--(B22);
\draw[ultra thin] (A22)--(B22);
\draw[ultra thin] (A23)--(B22);

\draw[ultra thin, <->] (8,0.4)-- node[right] {$\ell$} (8,2.1);

\end{tikzpicture}
\caption{The wedge $A\vee_{a,b}B$ and join $A+_\ell B$.}
\label{fig:wj}
\end{figure}
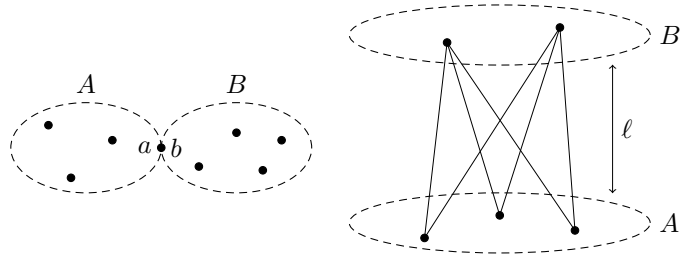

The two parts of the next lemma are proved by Leinster as Corollary 2.3.3 and Proposition 2.1.13 in \cite{LeinsterMagnitude2013}.

\begin{lemma}
\label{lem:wedgesum_join}
Let $A$ and $B$ be metric spaces which both have magnitude.
\begin{enumerate}
    \item \label{lem:wedgesum}
    % (\cite[Cor.~2.3.3]{LeinsterMagnitude2013})
    For each \(a \in A\) and \(b \in B\), the wedge $A\vee_{a,b}B$ has magnitude
    \[|A\vee_{a,b}| = |A|+|B|-1.\]
    \item \label{lem:join}
    % (\cite[Prop.~2.3.13]{LeinsterMagnitude2013})
    For each $ \ell \geq \frac{1}{2}\diam (A), \frac{1}{2}\diam (B)$, the join \(A +_\ell B\) has magnitude
    \[
    |A+_\ell B|=\frac{|A|+|B|-2e^{-\ell}|A||B|}{1-e^{-2\ell}|A||B|}.
    \]
\end{enumerate}
\end{lemma}

Now, denote by $K_n$ the metric space determined by the complete graph on $n$ vertices---that is, the $n$-point metric space with $d(x,x')=1$ whenever \(x \neq x'\). Then
\begin{equation}
\label{eq:Kn_mag}
    |t K_n|=\frac{n}{1+(n-1)e^{-t}}
\end{equation}
by \cite[Ex.~2.1.6]{LeinsterMagnitude2013}.

\begin{example}
\label{ex:willerton}
A metric space \(X\) is said to have \demph{the one-point property} if 
\[\lim_{t \to 0} |tX| = 1.\]
Joins of complete graphs provide essentially all known examples of spaces lacking the one-point property. Two in particular will be useful to us below.
\begin{enumerate}
\item \cite[Example 2.3.8]{LeinsterMagnitude2013}
Let $X=K_3+_1 2K_3$. Then $\lim_{t\to 0}|tX|=\sfrac{6}{5}$. 
\item \cite[Example 3.4]{roff2023small}
Let $X=\frac{4}{3}K_2+_1 2K_3$. Then $\lim_{t\to 0}|tX|=\sfrac{7}{6}$. 
\end{enumerate}
\end{example}

We can use \Cref{lem:wedgesum_join}, along with \Cref{ex:willerton}, to prove that magnitude fails absolutely to be continuous at a generic point of Gromov--Hausdorff space. That is, it fails to be generically continuous in the sense of Formulation A.

\begin{theorem}
\label{thm:nowhere}
    The function \(X \mapsto |X|\) is nowhere continuous on $\FMetmag$. 
\end{theorem}

\begin{proof}
    Take any $X\in\FMetmag$. Let $Y$ be a space without 
    the one-point property---for concreteness, we may choose $Y=K_3+_1 2K_3$ as in Example \ref{ex:willerton} (1). 
    Pick any points $x\in X$ and $y\in Y$, take \(t > 0\) and consider the wedge $X \vee tY = X\vee_{x, y} tY$. 
    Observe that $X$ embeds isometrically into $X\vee tY$, and that 
    in the Hausdorff metric on subspaces of the wedge, we have $d(X, X \vee tY ) \leq \diam(tY)\to 0$ 
    as $t\to 0$. It follows that $X\vee tY\to X$ in $\FMet$. 
    Meanwhile, for sufficiently small $t > 0$ the space $tY$ has magnitude \cite[Rmk.~2.1]{roff2023small}, so, by 
    \Cref{lem:wedgesum_join} \eqref{lem:wedgesum}, for sufficiently small $t > 0$ the space $X\vee tY$ has magnitude given by 
    $|X\vee tY|=|X|+|tY|-1$. Hence, as $t \to 0$, we have
    \[
    |X\vee tY|\to |X|+\lim_{t\to 0}|tY| - 1=|X|+\frac{1}{5}\neq |X|,
    \]
    which tells us that magnitude is not continuous at $X$. 
\end{proof}

In \cite[Thm.~3.8]{roff2023small}, it was shown that magnitude's failure of continuity at the one-point space can be `arbitrarily bad', in the sense that for each real number \(\alpha \geq 1\) there exists a finite metric space \(X\) with \(\lim_{t \to 0} |tX| = \alpha\). So far as we know, no example has yet been found of a finite metric space \(X\) such that \(\lim_{t \to 0} |tX| < 1\) or \(\lim_{t \to 0} |tX| = +\infty\). However, if we consider more general paths in \(\mathrm{FMet}\) converging to the one-point space, the limiting magnitude can attain any value in the extended reals. A modification of Example \ref{ex:willerton} realizes such paths.

\begin{example}
\label{ex:perturb}
Fix $\alpha\in[0, \infty)$. For sufficiently small $t > 0$, both 
\[
\left(t+\alpha t^2\right)K_3+_t 2t K_3 \ \ \ \text{ and } \ \ \
\left(\frac{4}{3}t+\alpha t^2\right)K_2+_t 2tK_3
\]
become metric spaces, and, as $t\to 0$, both these families converge to the one-point space. Using \Cref{lem:wedgesum_join} \eqref{lem:join} along with the expression \eqref{eq:Kn_mag} for \(|tK_n|\), we find that
\[
\lim_{t\to 0}\left|\left(t+\alpha t^2\right)K_3+_t 2tK_3\right|=\frac{6\alpha-6}{6\alpha-5}
\]
and
\[
\lim_{t\to 0}\left|\left(\frac{4}{3}t+\alpha t^2\right)K_2+_t 2tK_3\right|=\frac{9\alpha-14}{9\alpha-12}.
\]
By varying \(\alpha\), both limits can achieve arbitrary values in $\mathbb{R}\cup\{\infty\}$, except for 1.
\end{example}

\begin{figure}
\includegraphics[width=.5\textwidth]{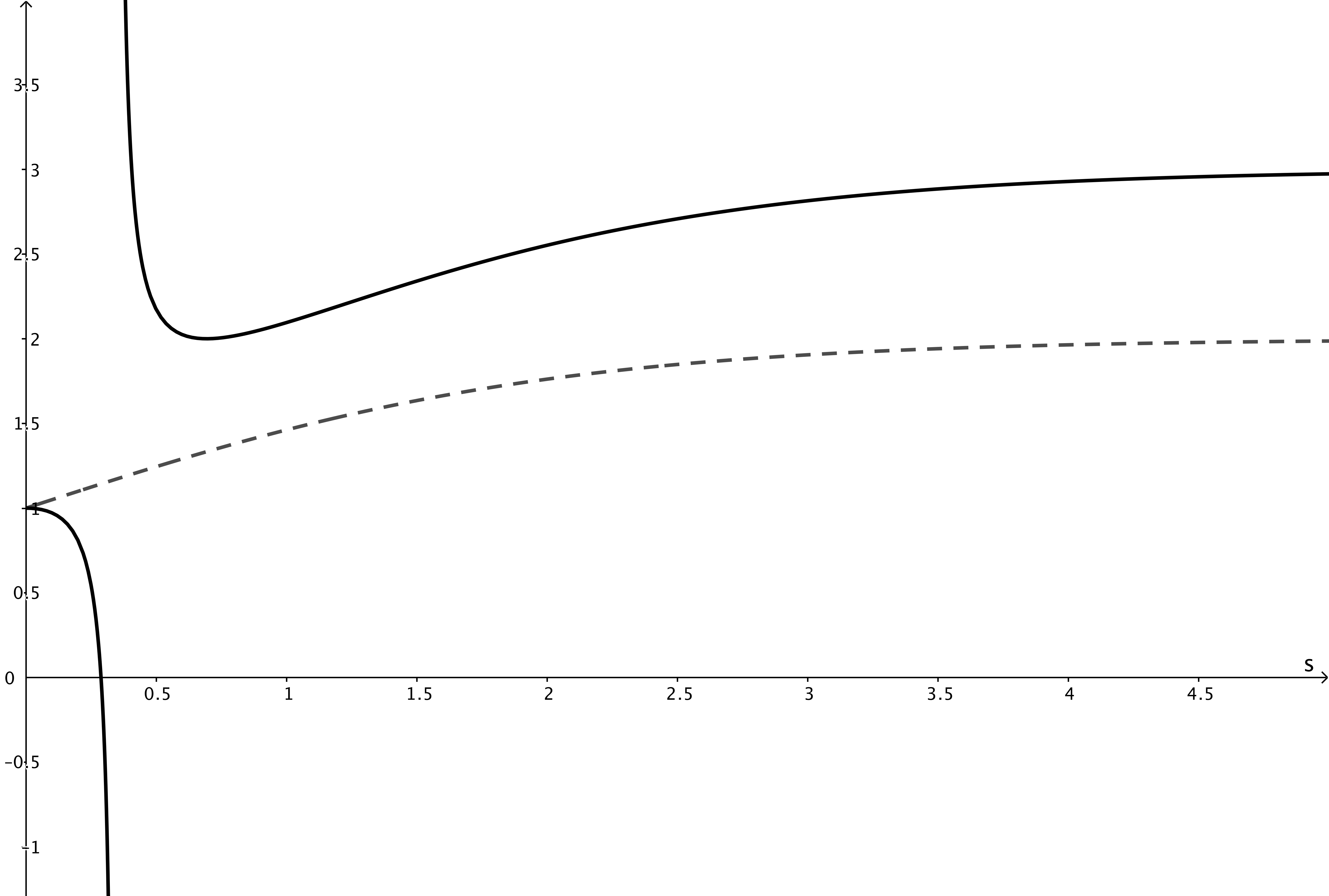}
\caption{The dashed graph shows the magnitude function of \(K_2\). The solid graph shows the pointwise limit, as \(t \to 0\), of the magnitude functions of the spaces \(X_t = \{\ast\} +_1 tY_{\alpha}\), for \(\alpha = 2\).}
\label{fig:K2_pw}
\end{figure}

Given a path \((X_t)_{0 < t \ll 1}\) in Gromov--Hausdorff space, we have so far discussed the behaviour, as \(t \to 0\), of the associated family of magnitudes \((|X_t|)_{0 < t \ll 1}\). One might wish to consider instead the behaviour of the associated family of {magnitude functions}, \((s \mapsto |sX_t|)_{0 < t \ll 1}\). In general this family of functions will not converge uniformly as \(t \to 0\), but we can always consider the pointwise limit and compare it to the magnitude function of the limiting space.

Take, for instance, the metric space $K_2$ with two points \(a,b\) and $d(a,b)=1$, whose magnitude function is given by \[|t K_2|=\frac{2}{1+e^{-t}}.\]
This example displays all the archetypal good behaviours of the magnitude function: it is smooth and monotonically increasing with \(t \in (0,\infty)\); we have \(\lim_{t \to 0}|tK_2| = 1\) and \(\lim_{t \to \infty}|tK_2| = 2\). It is the simplest example illustrating the idea that magnitude counts the `effective number of points' in a space at different scales: when the scale factor \(t\) is very small, \(K_2\) is `effectively' a one-point space; as \(t\) increases, its points gradually separate, until at last it is plainly a two-point space. 

However, given a path in $\FMet$ converging to \(K_2\), the pointwise limit of the associated magnitude functions need not have {any} of these good properties, as the following theorem shows.

\begin{theorem}
\label{thm:2points_fnc}
Given a path $(X_t)_{0 < t \ll 1}$ in \(\FMet\), let \(f^X \colon (0,\infty) \to \mathbb{R}\) denote the pointwise limit of their magnitude functions: \(f^X(s) = \lim_{t \to 0}|sX_t|\). Then:
\begin{enumerate}
\item For every \(\ell \in (1, \infty)\) there exists a path $(X_t)_{0 < t \ll 1}$ in $\FMet$ such that \(\lim_{t \to 0} X_t = K_2\), but \(\lim_{s \to 0} f^X(s) = \ell\). \label{2points_fnc_1}
\item For every \(L \in (2, \infty)\) there exists a path $(X_t)_{0 < t \ll 1}$ in $\FMet$ such that \(\lim_{t \to 0} X_t = K_2\), but \(\lim_{s \to \infty} f^X(s) = L\). \label{2points_fnc_2}
\item For every \(S \in (0,\infty)\) there exists a path $(X_t)_{0 < t \ll 1}$ in $\FMet$ such that \(\lim_{t \to 0} X_t = K_2\), but the function \(f^X\) has a singularity at \(S\). \label{2points_fnc_3}
\end{enumerate}
\end{theorem}

\begin{proof}
By \cite[Thm.~3.8]{roff2023small}, for every real number $\alpha \geq 1$ there exists a finite metric space $Y_\alpha$ such that $\lim_{t\to 0}|tY_\alpha|=\alpha$. Fix \(\alpha\) and let \(W_t = K_2 \vee t Y_\alpha\); then \(\lim_{t \to 0} W_t = K_2\). For each \(s > 0\), we have \(sW_t = sK_2 \vee st Y_\alpha\). It follows, by \Cref{lem:wedgesum_join} \eqref{lem:wedgesum}, that
\[f^W(s) = \lim_{t \to 0} |sW_t| = |sK_2| + \lim_{t \to 0} |st Y_\alpha| - 1 = |sK_2| + \alpha - 1,\]
from which we see that \(\lim_{s \to 0} f^W(s) = \alpha\) and \(\lim_{s \to \infty} f^W(s) = \alpha + 1\). Since \(\alpha\) can be chosen arbitrarily from \([1, \infty)\), this proves \eqref{2points_fnc_1} and \eqref{2points_fnc_2}.

Now, denote the one-point metric space by \(\{\ast\}\), and let $X_t= \{\ast\} +_1 (tY_\alpha)$. For sufficiently small \(t\) this join is a metric space, and \(\lim_{t\to 0} X_t =  K_2\). For each \(s > 0\), we have \(sX_t = \{\ast\} +_s (stY_\alpha)\). It follows, by \Cref{lem:wedgesum_join} \eqref{lem:join}, that
    \[
    f^X(s) = \lim_{t\to 0}| s X_t| = \frac{1+\alpha - 2e^{-s}\alpha}{1-e^{-2s}\alpha}, 
    \]
provided the denominator and the numerator are not both zero. Choosing \(\alpha = e^{2S}\) ensures that the function \(f^X\) has a singularity at \(S\), proving \eqref{2points_fnc_3}. (The case \(\alpha=2\) is shown in \Cref{fig:K2_pw}.)
\end{proof}

%%%%%%%%%%%%%%%%%%%%%%%%%%%%%%%%%%%%%%%%%

\section{Lines in Gromov--Hausdorff space}
\label{sec:lines}

In the remainder of the paper, we turn our attention to the second formulation of generic continuity:
\begin{formulationB}
    Let \(A\) and \(B\) be metric spaces. A function \(f \colon A \to B\) is generically continuous if, for a generic convergent sequence \((X_k)_{k \in \mathbb{N}}\) in \(A\), we have 
    \[\lim_{k \to \infty} f(X_k) = f\left(\lim_{k \to \infty} X_k \right).\]
\end{formulationB}
We cannot yet say definitively whether or not magnitude is generically continuous in this sense. Instead, we will restrict attention to a particular class of paths in Gromov--Hausdorff space---those that can be thought of as straight lines---and we will prove that for a generic line \((Y_t)_{0 < t \leq 1}\) we have
\begin{equation}
\label{eq:mag_pres}
    \lim_{t \to 0} |Y_t| = |\lim_{t \to 0} Y_t|.
\end{equation}
To make this precise, in this section we define what it means to be a line in Gromov--Hausdorff space. We show that, just as each finite metric space can be described by an equivalence class of matrices (its `distance matrix'), each line can be described by an equivalence class of block matrices. That description allows us to equip the set of Gromov--Hausdorff lines with a topology. In the next section we will show that for a generic space \(X\), the space of all lines converging to \(X\) contains a dense open subset on which property \eqref{eq:mag_pres} holds.

\begin{definition}
\label{def:GH_lines}
    Let \((X,d_X)\) and \((Y,d_Y)\) be finite metric spaces, and let \(f\colon Y \surj X\) be a surjective function. The \demph{line from \(Y\) to \(X\) along \(f\)} is the family of metric spaces \(\{(Y, d_{f,t}) \mid 0 < t \leq 1\}\) defined by
    \[d_{f,t}(y,y') = td_Y(y,y') + (1-t)d_X(f(y),f(y')).\]
    By construction we have \((Y, d_{f,1}) = (Y, d_Y)\), and
    \[\lim_{t \to 0} (Y, d_{f,t}) = (X,d_X)\]
    in the Gromov--Hausdorff topology. We will often denote \((Y, d_{f,t})\) by \(Y_{f,t}\). 
\end{definition}

\begin{figure}[htbp]
\centering
\begin{tikzpicture}

\coordinate (Y11) at (-1,3); 
\coordinate (Y12) at (0,3.5); 
\coordinate (Y13) at (0.5,2.8); 

\coordinate (Y21) at (1.4,3); 

\coordinate (Y31) at (2,3.7); 
\coordinate (Y32) at (4.2,3); 

\coordinate (Z1) at (0,0); 
\coordinate (Z2) at (2,0); 
\coordinate (Z3) at (4,0);

\coordinate (X11) at (-0.5,1.5); 
\coordinate (X12) at (0,1.75); 
\coordinate (X13) at (0.25,1.4); 

\coordinate (X21) at (1.7,1.5); 

\coordinate (X31) at (3,1.85); 
\coordinate (X32) at (4.1,1.5);

\filldraw[fill=black, draw=black] (Y11) circle[radius=0.05] node[above] {$y_i^1$};
\filldraw[fill=black, draw=black] (Y12) circle[radius=0.05] node[right] {$y_i^2$};
\filldraw[fill=black, draw=black] (Y13) circle[radius=0.05] node[anchor=south west] {$y_i^3$};
\filldraw[fill=black, draw=black] (Y21) circle[radius=0.05];
\filldraw[fill=black, draw=black] (Y31) circle[radius=0.05];
\filldraw[fill=black, draw=black] (Y32) circle[radius=0.05];
\filldraw[fill=black, draw=black] (Z1) circle[radius=0.05] node[right] {$x_i$};
\filldraw[fill=black, draw=black] (Z2) circle[radius=0.05];
\filldraw[fill=black, draw=black] (Z3) circle[radius=0.05];

\filldraw[fill=black, draw=black] (X11) circle[radius=0.05];
\filldraw[fill=black, draw=black] (X12) circle[radius=0.05];
\filldraw[fill=black, draw=black] (X13) circle[radius=0.05];
\filldraw[fill=black, draw=black] (X21) circle[radius=0.05];
\filldraw[fill=black, draw=black] (X31) circle[radius=0.05];
\filldraw[fill=black, draw=black] (X32) circle[radius=0.05];

\draw[ultra thin, gray] (Y11)--(Z1);
\draw[ultra thin, gray] (Y12)--(Z1);
\draw[ultra thin, gray] (Y13)--(Z1);
\draw[ultra thin, gray] (Y21)--(Z2);
\draw[ultra thin, gray] (Y31)--(Z3);
\draw[ultra thin, gray] (Y32)--(Z3);

\draw[densely dashed] (2,0) circle [x radius=3,y radius=0.3];
\draw (-1,0) node[left] {$X$};

\draw[densely dashed] (2,1.6) circle [x radius=3.5,y radius=0.4];
\draw (-1.5,1.6) node[left] {$Y_t$};

\draw[densely dashed] (2,3.24) circle [x radius=4,y radius=0.7];
\draw (-2,3.24) node[left] {$Y=Y_1$};

\end{tikzpicture}
\caption{A line in Gromov--Hausdorff space from \(Y\) to \(X\).}
\label{fig:linGH}
\end{figure}
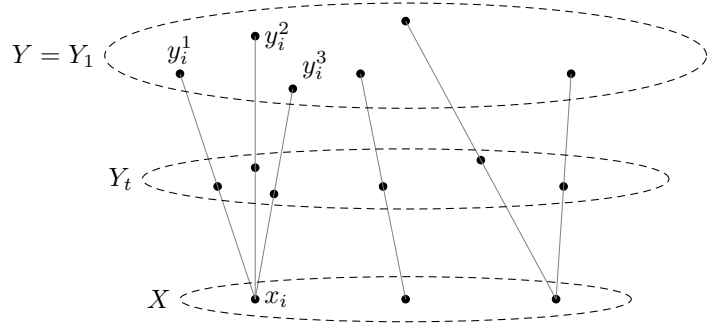

\begin{example}
    There is a unique line from each \(Y \in \mathrm{FMet}\) to the one-point space: it is the family of scaled spaces \((tY)_{0 < t \leq 1}\) considered in the study of the one-point property (and throughout the study of magnitude).
\end{example}

\begin{example}
    The family of spaces \((X \vee_{x,y} tY)_{0 < t \leq 1}\) constructed in the proof of \Cref{thm:nowhere} is the line from \(X \vee_{x,y} Y\) to \(X\) along the function \(f\) specified by
    \[f(u) = \begin{cases}
        u & u \in X \\
        x \sim y & u \in Y.
    \end{cases}\]
    The paths to \(K_2\) described in the proof of \Cref{thm:2points_fnc} are also lines. The paths constructed in \Cref{ex:perturb} each converge to the one-point space, but they are not lines.
\end{example}

Since isometric spaces represent the same point in Gromov--Hausdorff space, two lines \((Y_{f,t})\) and \((W_{g,t})\) in \(\FMet\) are indistinguishable if there are isometries \(Y_{f,t} \simeq W_{g,t}\) for each \(t\). The next lemma, specifying when two surjective functions represent the same line, follows immediately from \Cref{def:GH_lines}.

\begin{lemma}
\label{lem:same_line}
    Let \(X\) and \(Y\) be finite metric spaces, and let \(f,g \colon Y \surj X\) be surjective functions. Suppose there exist isometries \(h \colon Y \xrightarrow{\sim} Y\) and \(k \colon X \xrightarrow{\sim} X\) such that \(k \circ f = g \circ h\). Then \(h\) gives an isometry \(Y_{f,t} \xrightarrow{\sim} Y_{g,t}\) for every \(t\). \qed
\end{lemma}

In light of the lemma, we make the following definition.

\begin{definition}
\label{def:set_of_lines}
    Let \(X\) be a finite metric space. The 
    \demph{set of lines to \(X\)} 
    is the set
    \[\lines{}{X} = \left\{(Y_{f,t})_{0 < t \leq 1} \mid f \colon Y \surj X \text{ is a surjective function}\right\} / \sim \]% 
    where \((Y_{f,t}) \sim (Y_{g,t})\)
    if there exist isometries \(h \colon Y \xrightarrow{\sim} Y\) and \(k \colon X \xrightarrow{\sim} X\) such that \(k \circ f = g \circ h\).
\end{definition}

In order to make a statement about the generic behaviour of magnitude with respect to lines, we need to equip the set of lines with a topology. We will topologise \(\lines{}{X}\) by relating it to a certain space of matrices.

Recall that every finite metric space \(X\), once we label its points as \(x_1,\ldots,x_m\), is determined by its \demph{distance matrix} 
\[D_X =  (d_X(x_i,x_j))_{i,j=1}^m.\]
Thus, ordered \(m\)-point metric spaces are parametrized by the set
\[
\label{eq:matrix}
\Met_m=
\left\{
(d_{ij})_{i, j=1, \dots, m}
\left|\ 
\begin{aligned}
    &d_{ii}=0\ (1\leq i\leq m)\\
&d_{ij}=d_{ji}>0\ (1\leq i<j\leq m)\\
&d_{ij}+d_{jk}\geq d_{ik}\ (1\leq i, j, k\leq m)
\end{aligned}
\right.
\right\},
\]
which is a $\binom{m}{2}$-dimensional convex cone in Euclidean space. The set of (unordered) \(m\)-point metric spaces is parametrized by \(\mathrm{Met}_m / \mathfrak{S}_m\), where the action of the group \(\mathfrak{S}_m\) permutes the rows and columns of each matrix in \(\mathrm{Met}_m\).

In \cite[\S 2]{roff2023small}, the small-scale limit of the magnitude function was studied via this identification of finite metric spaces with equivalence classes of matrices. To adapt the techniques of that paper to the present setting, we will identify lines in Gromov--Hausdorff space with certain equivalence classes of block matrices.

\begin{proposition}
\label{prop:space_of_lines}
    Let \(X\) be a finite metric space. Label its points as \(x_1,\ldots, x_m\), and let \(\mathrm{Iso}(X)\) denote the group of isometries of \(X\), regarded as a subgroup of the permutation group \(\mathfrak{S}_m\).
    There is a bijection
    \begin{equation}
    \label{eq:lines_biject}
        \lines{}{X} 
        \xrightarrow{\sim}
        \left(\bigsqcup_{\substack{(r_1,\ldots, r_m) \\ \sum r_i = n \\ n, r_1 \ldots, r_m \geq 1}} \mathrm{Met}_n/(\mathfrak{S}_{r_1} \times \cdots \times \mathfrak{S}_{r_m})\right) / \mathrm{Iso}(X).
    \end{equation}
    Here, each tuple \((r_1, \ldots, r_m)\) partitions the rows and columns of each matrix \(D \in \mathrm{Met}_n\). The action of \(\mathfrak{S}_{r_i}\) permutes the rows and columns within the \(i^\th\) section of the partition, and the action of \(\mathrm{Iso}(X)\) permutes the sections of the partition.
\end{proposition}

\begin{proof}
    First, let \(Y = (\{y_1,\ldots, y_n\}, d_Y)\), and consider a line from \(Y\) to \(X\) along some surjection \(f \colon \{y_1,\ldots, y_n\} \surj \{x_1,\ldots, x_m\}\). For each \(i \in \{1,\ldots, m\}\), let \(r_i = \# f^{-1}(x_i)\). Let the points of \(f^{-1}(x_i)\) inherit their order from \(Y\), but relabel them so that
    \(f^{-1}(x_i) = \{y_i^1,\ldots, y_i^{r_i}\}\).
    In this way, \(f\) is specified by the labelling
    \((y_1^1, \ldots y_1^{r_1}, \ldots, y_m^{1}, \ldots, y_m^{r_m})\)
    of the points of \(Y\), and hence by a block matrix  whose columns and rows are partitioned by the tuple \((r_1,\ldots, r_m)\).
    
    This defines a function 
    \begin{equation}
    \label{eq:space_of_lines0}
        \phi \colon \{f \colon (\{y_1,\ldots, y_n\} ,d_Y) \surj (\{x_1,\ldots, x_m\}, d_X)\} \to \bigsqcup_{\substack{(r_1,\ldots, r_m) \\ \sum r_i = n \\ n, r_1 \ldots, r_m \geq 1}} \mathrm{Met}_n
    \end{equation}
    which is obviously injective. To see that \(\phi\) is surjective, take any pair \((D, \mathbf{r})\) of a distance matrix \(D\) and \(\mathbf{r} = (r_1,\ldots, r_m)\), where \(r_i \geq 1\) satisfy \(\sum r_i = n\). Let 
    \[Y = (\{y_1^1,\ldots, y_1^{r_1}, \ldots, y_m^1, \ldots y_m^{r_m}\},d)\]
    be the metric space represented by \(D\) and specify \(f \colon Y \surj X\) by \(f(y_i^\alpha) = x_i\); then \(\phi(Y,f) = (D,\mathbf{r})\). This shows that \(\phi\) is a bijection; we need to see that it descends to a bijection after taking the appropriate quotient on each side of \eqref{eq:space_of_lines0}.
    
    By \Cref{def:set_of_lines}, surjections \(f, g \colon (\{y_1,\ldots, y_n\} ,d_Y) \surj (\{x_1,\ldots, x_m\}, d_X)\) represent the same line to \(X\) if and only if there exist \(\sigma \in \mathfrak{S}_n\) and \(\rho \in \mathfrak{S}_m\) such that the induced maps \(\sigma\colon Y \to Y\) and \(\rho \colon X \to X\) are both isometries---so, in particular,
    \begin{equation}
    \label{eq:space_of_lines1}
        d_Y(y_{\sigma(i)}, y_{\sigma(j)}) = d_Y(y_i,y_j)
    \end{equation}
    for all \(i,j \in \{1,\ldots, n\}\)---
    and we have 
    \begin{equation}
    \label{eq:space_of_lines2}
        \rho \circ f = g \circ \sigma.
    \end{equation}
    Condition \eqref{eq:space_of_lines1} says precisely that \(\sigma\) relates the matrix \(\phi(f)\) to \(\phi(g)\) by permuting its rows and columns.
    By \eqref{eq:space_of_lines2}, \(\sigma\) maps 
    \(\{y_i^1, \dots, y_i^{r_i}\}\) bijectively to 
    \(\{y_{\rho(i)}^1, \dots, y_{\rho(i)}^{r_{\rho(i)}}\}\). 
    Hence \(r_i=r_{\rho(i)}\), and there exists 
    \(\eta_i\in\mathfrak{S}_{r_i}\) such that 
    \(\sigma(y_i^\alpha)=y_{\rho(i)}^{\eta_i(\alpha)}\). 
    Therefore, \(f\) and \(g\) are related by 
    permutations \(\eta_i\in\mathfrak{S}_{r_i}\), 
    \(i=1, \dots, m\), and \(\rho\in\mathfrak{S}_m\). 
    This means that \(f\) and \(g\) are in the same 
    equivalence class of the right-hand side of 
    \eqref{eq:lines_biject}. The converse is obvious.
\end{proof}

We equip \(\mathrm{Met}_n\) with the Euclidean topology, and let this descend to the quotient on each component in the coproduct on the right of 
\eqref{eq:lines_biject}, then to the quotient by \(\mathrm{Iso}(X)\). The bijection in \eqref{eq:lines_biject} then makes \(\lines{}{X}\) into a topological space, which we will call \demph{the space of lines to \(X\)}.

\begin{remark}
Some readers may prefer to think of a line to \(X\) as representing a \demph{direction of approach} to \(X\): lines \((Y_t)\) and \((Y_t')\) approach `from the same direction' if \(Y_1' \simeq Y_t\) for some \(0<t\leq 1\), or vice versa. Our main result (\Cref{thm:main}) then says that for a generic finite metric space \(X\), and a linear sequence approaching \(X\) from a generic direction, magnitude behaves continuously. Thus, for continuity of magnitude, the bad directions are rare, and almost all directions are good.
\end{remark}

%%%%%%%%%%%%%%%%%%%%%%%%%%%%%%%%%%%%%%%%%
\section{Magnitude behaves continuously for generic limits along lines}
\label{sec:conti}

The goal in this section is to prove that for a generic line \((Y_{f,t})\) in Gromov--Hausdorff space, the limit as \(t \to 0\) is preserved by magnitude, meaning that
\[\lim_{t \to 0} |Y_{f,t}| = |\lim_{t \to 0} Y_{f,t}|.\]
Our main theorem runs as follows.

\begin{theorem}
\label{thm:main}
    For any finite metric space \(X\) such that \(Z_X\) is invertible, the space of lines to \(X\) contains a dense open subset whose elements satisfy \(\lim_{t \to 0} |Y_{f,t}| = |X|\).
\end{theorem}

The proof of \Cref{thm:main} depends on two propositions, whose proofs occupy most of the section. For ease of comprehension, we will first prove \Cref{thm:main} \emph{assuming} Propositions \ref{prop:poly_exists} and \ref{prop:poly_nonzero}, then proceed to the proofs of the propositions.

\begin{proof}[Proof of \Cref{thm:main}]
Let \(X\) be a metric space such that \(\det(Z_X) \neq 0\). Suppose \(\# X = m\). Propositions \ref{prop:poly_exists} and \ref{prop:poly_nonzero} tell us there exists, for each \(\mathbf{r} = (r_1, \ldots, r_m)\), a non-zero polynomial \(F_{X,\mathbf{r}}\) defined on \(\mathrm{Met}_n\) (where \(n=\sum r_i\)), whose non-vanishing at a matrix representing a line \((Y_{f,t})_{0 < t \leq 1}\) guarantees that \(\lim_{t \to 0} |Y_{f,t}| = |X|\). The complement of the zero set of \(F_{X,\mathbf{r}}\) is, therefore, a dense open subset of \(\mathrm{Met}_{n}\) on which this property holds. Since the quotient of \(\mathrm{Met}_n\) by the action of the group \(\mathfrak{S}_{r_1} \times \cdots \times \mathfrak{S}_{r_m}\) is an open map, the image of that subset is dense and open in the \(\mathbf{r}\)-component of the coproduct 
\[\bigsqcup_{\substack{(r_1,\ldots, r_m) \\ \sum r_i = n \\ n, r_1 \ldots, r_m \geq 1}} \mathrm{Met}_n/(\mathfrak{S}_{r_1} \times \cdots \times \mathfrak{S}_{r_m}).\]
And since the quotient of this space by the action of \(\mathrm{Iso}(X)\) is an open map, the union of the zero sets descends to a dense open subset of \(\lines{}{X}\).
\end{proof}

\begin{remark}
    If \(X\) is such that \(\det(Z_X) \neq 0\), then every line along a bijective function \(f \colon Y \xrightarrow{\sim} X\) satisfies \(\lim_{t \to 0} |Y_{f,t}| = |X|\). Indeed, in this case we have 
    \[Z_X = \lim_{t \to 0} Z_{Y_{f,t}}\]
    in \(\mathbb{R}^{m \times m}\), so the continuity of the determinant implies that \(Z_{Y_{f,t}}\) is invertible for sufficiently small \(t\), and that
    \(Z_X^{-1} = \lim_{t \to 0} Z_{Y_{f,t}}^{-1}\).
    Thus, we can use the expression \eqref{eq:mag_def} for magnitude to see immediately that \(|X| = \lim_{t \to 0} |Y_{f,t}|\).
    
    As we will see, the technicalities become much more delicate when we consider lines along surjections \(f \colon Y \surj X\) which are not injective---that is, linear limits in which distinct points become united.
\end{remark}

To establish the existence of the polynomial \(F_{X,\mathbf{r}}\), we will make use of a few familiar facts of linear algebra, which we gather here for convenience. Recall that an \demph{elementary column operation} on a matrix \(A\) is the operation of adding a multiple of one column to another. Two matrices $A$ and $B$ (of the same dimensions) are said to be \demph{column-equivalent} if $B$ can be obtained from $A$ by applying finitely many elementary column operations. In this case we write $A\sim B$. This notion is motivated by the following fundamental facts.

\begin{lemma}
\label{lem:elem1}
    Let $A$ be an $n\times n$ matrix. 
    \begin{enumerate}
        \item If \(B\) is an \(n \times n\) matrix such that \(A \sim B\), then \(\det(A) = \det(B)\). \label{lem:elem1_pt3}
        \item If $\det(A)\neq 0$, then $A$ is column-equivalent to a diagonal matrix with diagonal entries \(\alpha_1, \ldots, \alpha_n\), and $\det(A)=\prod_{i=1}^n\alpha_i$.  \label{lem:elem1_pt1}
        \item If $\det(A)=0$, $A$ is column-equivalent to a matrix with a zero column. \label{lem:elem1_pt2}
    \end{enumerate}
\end{lemma}

The next lemma is equally elementary, though slightly less standard.

\begin{lemma}
\label{lem:elem2}
    Let $A=(a_{ij})$ and $B=(b_{ij})$ be $m\times n$ matrices such that $A\sim B$. Let $\bm{1}_r=(1,1,\cdots, 1)^\top$ denote the $r$-dimensional column vector with all entries $1$. Given positive integers $r_1, \dots, r_m$, define $\left(\sum_{i=1}^m r_i\right)\times n$ matrices
        \[
        \widetilde{A} =
        \begin{pmatrix}
            a_{11}\bm{1}_{r_1}&a_{12}\bm{1}_{r_1}&\dots&a_{1n}\bm{1}_{r_1}\\
            a_{21}\bm{1}_{r_2}&a_{22}\bm{1}_{r_2}&\dots&a_{2n}\bm{1}_{r_2}\\
            \vdots&\vdots&\ddots&\vdots\\
            a_{m1}\bm{1}_{r_m}&a_{m2}\bm{1}_{r_m}&\dots&a_{mn}\bm{1}_{r_m}
        \end{pmatrix} \]
        and
        \[\widetilde{B} =
        \begin{pmatrix}
            b_{11}\bm{1}_{r_1}&b_{12}\bm{1}_{r_1}&\dots&b_{1n}\bm{1}_{r_1}\\
            b_{21}\bm{1}_{r_2}&b_{22}\bm{1}_{r_2}&\dots&b_{2n}\bm{1}_{r_2}\\
            \vdots&\vdots&\ddots&\vdots\\
            b_{m1}\bm{1}_{r_m}&b_{m2}\bm{1}_{r_m}&\dots&b_{mn}\bm{1}_{r_m}
        \end{pmatrix}.\]
        Then $\widetilde{A}\sim\widetilde{B}$.
\end{lemma}

From Lemmas \ref{lem:elem1} and \ref{lem:elem2} one can derive the following statement, which is the one that we will actually use.

\begin{lemma}
\label{lem:elem3}
    Let $n>0$ and $r_1, \dots, r_m>0$, $n=\sum_{i=1}^m r_i$ be 
    positive integers. Let $A=(a_{ij})$ be an $m\times m$ 
    matrix and $M$ be an $n\times n$ matrix of the form 
    \[
    M=
    \begin{pmatrix}
        a_{11}\bm{1}_{r_1}& a_{12}\bm{1}_{r_1}&\dots& a_{1n}\bm{1}_{r_1}&*&\dots&*\\
        a_{21}\bm{1}_{r_2}& a_{22}\bm{1}_{r_2}&\dots& a_{2n}\bm{1}_{r_2}&*&\dots&*\\
        \vdots & \vdots &\ddots& \vdots& \vdots&\ddots& \vdots\\
        a_{m1}\bm{1}_{r_m}& a_{m2}\bm{1}_{r_m}&\dots& a_{mm}\bm{1}_{r_m}&*&\dots&*
    \end{pmatrix}.
    \]
    Then 
\begin{equation}
\label{eq:detM}
    \det(M)=\det(A)\cdot\det
    \begin{pmatrix}
        \bm{1}_{r_1}& 0&\dots& 0&*&\dots&*\\
        0& \bm{1}_{r_2}&\dots& 0&*&\dots&*\\
        \vdots & \vdots &\ddots& \vdots& \vdots&\ddots& \vdots\\
        0& 0&\dots& \bm{1}_{r_m}&*&\dots&*
    \end{pmatrix}.
\end{equation}
\end{lemma}
\begin{proof}
    Suppose $\det(A)=0$. Then by Lemma \ref{lem:elem1} \eqref{lem:elem1_pt2} and 
    Lemma \ref{lem:elem2}, $M$ is column-equivalent to 
    a matrix with a zero column, so in this case, both sides of (\ref{eq:detM}) 
    are zero. If $\det(A)\neq 0$, then by Lemma \ref{lem:elem1} \eqref{lem:elem1_pt1} 
    and 
    Lemma \ref{lem:elem2}, $M$ is column-equivalent to 
    \[
        \begin{pmatrix}
        \alpha_1\bm{1}_{r_1}& 0&\dots& 0&*&\dots&*\\
        0& \alpha_2\bm{1}_{r_2}&\dots& 0&*&\dots&*\\
        \vdots & \vdots &\ddots& \vdots& \vdots&\ddots& \vdots\\
        0& 0&\dots& \alpha_n\bm{1}_{r_n}&*&\dots&*
    \end{pmatrix}, 
    \]
    with $\det(A)=\alpha_1\cdots\alpha_n$. Then Lemma 
    \ref{lem:elem1} \eqref{lem:elem1_pt3} yields (\ref{eq:detM}). 
\end{proof}

We can now proceed to the two propositions on which \Cref{thm:main} depends. Throughout what follows, we adopt the following notational conventions:
\begin{itemize}
    \item Given a matrix $A=(a_{ij})_{ij}$ (respectively, a vector $\bm{a}=(a_i)_i$), we denote the entry-wise $p$-th power by $A^{\langle p\rangle}=(a_{ij}^p)_{ij}$ (respectively, $\bm{a}^{\langle p\rangle}=(a_i^p)_i$). 
    \item Given a non-negative integral vector $\bm{p}=(p_1, \dots, p_n)$, we write $\bm{p}!=\prod_{i}p_i!$ and $|\bm{p}|=\sum_{i}p_i$. 
    \item The symbol $\runij$ will refer to an index running over $1, \dots, n$, while $\runab$ runs over $1, \dots, r_i$ (or $r_j$). 
    \item $\bm{1}_r=(1,1,\cdots, 1)^\top$ will denote the column vector of length \(r\) with all entries $1$, and $\bm{1}_{m, n}$ will denote the $m\times n$ matrix with all entries $1$. 
\end{itemize}

\begin{proposition}
\label{prop:poly_exists}
Let \(X\) be an \(m\)-point metric space such that \(\det(Z_X) \neq 0\). For each \(\mathbf{r} = (r_1,\ldots,r_m)\) there exists a polynomial
\[
F_{X,\mathbf{r}} \in \mathbb{R}[a_{ij}^{\alpha\beta}]
\]
(where $i,j= 1, \ldots, m$; $1\leq \alpha\leq r_i$; and $1\leq\beta\leq r_j$) with the following property.
If \(f \colon Y \surj X\), represented by the ordering \((y_1^1, \ldots y_1^{r_1}, \ldots, y_m^{1}, \ldots, y_m^{r_m})\) of \(Y\), satisfies
\[F_{X,\mathbf{r}} (d_Y(y_{i}^{\alpha}, y_j^\beta))\neq 0,\]
 then 
\[\lim_{t\to 0}|Y_{f,t}|=|X|.\]
\end{proposition}

\begin{proof}
Throughout this proof, let \(Z_{f,t}\) denote the similarity matrix of \(Y_{f,t}\), and, given an \(n \times n\) matrix \(M\), let \(\sumofentry(M) = \sum_{i,j=1}^n M(i,j)\).

Using the expression for magnitude in equation \eqref{eq:mag_def}, we have
\begin{equation}
\label{eq:magXt}
    |Y_{f,t}|=\frac{\sumofentry (\adj (Z_{f,t}))}{\det (Z_{f,t})}
\end{equation}
whenever \(\det(Z_{f,t}) \neq 0\). For each \(y,y' \in Y_{f,t}\), the \((y,y')\)-entry of \(Z_{f,t}\) is
\[e^{-d_X(f(y),f(y'))} \cdot e^{-t(d_Y(y,y') - d_X(f(y),f(y')))};\]
hence, \(\det(Z_{f,t})\) is a \emph{generalized polynomial} in \(e^{-t}\), in the sense of \cite[\S 2]{roff2023small} (but now with real coefficients). It follows that \(\det(Z_{f,t})\) vanishes for at most finitely many \(t\). In particular, for all sufficiently small \(t > 0\) we must have \(\det(Z_{f,t}) \neq 0\), so the expression \eqref{eq:magXt} is valid for \(0 < t \ll 1\). We will compute the denominator and the numerator of \eqref{eq:magXt} separately, beginning with the denominator.

First, we introduce some compressed notation. For \(i,j \in \{1, \ldots, m\}\) and \(\alpha \in \{1,\ldots, r_i\}\), \(\beta \in \{1, \ldots, r_j\}\), let \(d_{ij} = d_X(x_i,x_j)\) and let
\[c_{ij}^{\alpha\beta} =d_Y(y_{i}^{\alpha}, y_{j}^{\beta})-d_X(x_i,x_j).\]
Then 
\[
d_{Y_{f,t}}(y_{i}^{\alpha}, y_{j}^{\beta})=
d_X(x_i,x_j) +tc_{ij}^{\alpha\beta}. 
\]
Let us define for each \(i,j \in \{1, \ldots, m\}\)
an $r_i\times r_j$ matrix 
\[
\Gamma_{ij}{(t)}
=
(\exp(-t c_{ij}^{\alpha\beta}))_{
\substack{\alpha=1, \dots, r_i\\
\beta=1, \dots, r_j}}. 
\]
Using these matrices, one can write \(Z_{f,t}\) as the block matrix
\[
Z_{f,t}=
(e^{-d_{ij}} \cdot \Gamma_{ij}(t))_{
i, j=1, \dots, m}. 
\]
Expanding the entries of \(\Gamma_{ij}{(t)}\) as power series in \(t\), we see that the $(\alpha, \beta)$ 
entry of the $(i, j)$ block of \(Z_{f,t}\) is 
\begin{equation}
\label{eq:exponential}
e^{-d_{ij}}\cdot e^{-t c_{ij}^{\alpha\beta}} =
e^{-d_{ij}}\sum_{p\geq 0}\frac{(-tc_{ij}^{\alpha\beta})^p}{p!}. 
\end{equation}

We can now examine the denominator of \eqref{eq:magXt}. Using the multilinearity of the determinant for columns, we have the following expansion of $\det(Z_{f,t})$:
\[
\sum_{\bm{p}_{\runij}^{\runab}}
\frac{(-t)^{|\bm{p}_{\runij}^{\runab}|}}{\bm{p}_{\runij}^{\runab}!}
\det
\begin{pmatrix}
    e^{-d_{11}}(\bm{c}_{11}^{\runab 1})^{\langle p_1^1\rangle}&
    e^{-d_{11}}(\bm{c}_{11}^{\runab 2})^{\langle p_1^2\rangle}&
    \dots&
    e^{-d_{1m}}(\bm{c}_{1m}^{\runab r_m})^{\langle p_m^{r_m}\rangle}\\
    e^{-d_{21}}(\bm{c}_{21}^{\runab 1})^{\langle p_1^1\rangle}&
    e^{-d_{22}}(\bm{c}_{21}^{\runab 2})^{\langle p_1^2\rangle}&
    \dots&
    e^{-d_{2m}}(\bm{c}_{2m}^{\runab r_m})^{\langle p_m^{r_m}\rangle}\\
    \vdots&\vdots&\ddots&\vdots\\
    e^{-d_{m1}}(\bm{c}_{m1}^{\runab 1})^{\langle p_1^1\rangle}&
    e^{-d_{m2}}(\bm{c}_{m1}^{\runab 2})^{\langle p_1^2\rangle}&
    \dots&
    e^{-d_{nn}}(\bm{c}_{mm}^{\runab r_m})^{\langle p_m^{r_m}\rangle}
\end{pmatrix}, 
\]
where $\bm{p}_{\runij}^{\runab}=(p_i^\alpha)_{\substack{i=1, \dots, m\\ \alpha=1, \dots, r_i}}$ runs over vectors of length \(n = \sum r_i\) with entries in \(\mathbb{Z}_{\geq 0}\).

Let us call the block 
\[
\begin{pmatrix}
    e^{-d_{1j}}(\bm{c}_{1j}^{\runab 1})^{\langle p_j^1\rangle}&
    e^{-d_{1j}}(\bm{c}_{1j}^{\runab 2})^{\langle p_j^2\rangle}&
    \dots&
    e^{-d_{1j}}(\bm{c}_{1j}^{\runab r_j})^{\langle p_j^{r_j}\rangle}\\
    e^{-d_{2j}}(\bm{c}_{2j}^{\runab 1})^{\langle p_j^1\rangle}&
    e^{-d_{2j}}(\bm{c}_{2j}^{\runab 2})^{\langle p_j^2\rangle}&
    \dots&
    e^{-d_{2j}}(\bm{c}_{2j}^{\runab r_j})^{\langle p_j^{r_j}\rangle}\\
    \vdots&\vdots&\ddots&\vdots\\
    e^{-d_{nj}}(\bm{c}_{mj}^{\runab 1})^{\langle p_j^1\rangle}&
    e^{-d_{nj}}(\bm{c}_{mj}^{\runab 2})^{\langle p_j^2\rangle}&
    \dots&
    e^{-d_{nj}}(\bm{c}_{mj}^{\runab r_j})^{\langle p_j^{r_j}\rangle}
\end{pmatrix}
\]
%$(e^{-d_{\runij j}}c_{\runij j}^{\runab\runab})$ 
the $j$-th column block.  
If $p_j^\beta=0$, then the $\beta$-th column of the $j$-th 
column block is 
\[
\begin{pmatrix}
    e^{-d_{1j}}\bm{1}_{r_1}\\
    e^{-d_{2j}}\bm{1}_{r_2}\\
    \vdots\\
    e^{-d_{mj}}\bm{1}_{r_m}
\end{pmatrix}. 
\]
If two of $p_j^1, p_j^2, \dots, p_j^{r_j}$ are zero, the corresponding columns are identical, so it is enough to consider 
$p_{\runij}^{\runab}$ such that at most one of $p_j^1, \dots, p_j^{r_j}$ is zero for each $j=1, \dots, m$. 
Hence, the lowest degree term appearing with non-zero coefficient is 
\[
(-t)^{\sum_{i=1}^m (r_i -1)} = (-t)^{n - m}
\]
which is attained by 
the indices of the form $\bm{p}_j^{\runab}=(1, \dots, 1, 0, 1, \dots, 1)$.
The coefficient of $(-t)^{n-m}$ is 
\begin{equation}
\label{eq:coeff}
\sum_{\substack{\beta_1=1, \dots, r_1\\ \cdots\\ \beta_m=1, \dots, r_m}}
\det
\begin{pmatrix}
    \dots&
    e^{-d_{1j}}(\bm{c}_{1j}^{\runab 1})&
    \dots&
    \stackrel{\beta_j}{\widehat{e^{-d_{1j}}(\bm{1}_{r_1})}}&
    \dots&
    e^{-d_{1j}}(\bm{c}_{1j}^{\runab r_j})&
    \dots\\
    \dots&
    e^{-d_{2j}}(\bm{c}_{2j}^{\runab 1})&
    \dots&
    e^{-d_{2j}}(\bm{1}_{r_2})&
    \dots&
    e^{-d_{2j}}(\bm{c}_{2j}^{\runab r_j})&
    \dots\\
    &\vdots&\ddots&\vdots&\ddots&\vdots\\
    \dots&
    e^{-d_{nj}}(\bm{c}_{mj}^{\runab 1})&
    \dots&
    e^{-d_{nj}}(\bm{1}_{r_m})&
    \dots&
    e^{-d_{nj}}(\bm{c}_{mj}^{\runab r_j})&
    \dots
\end{pmatrix}. 
\end{equation}
Here, the sum is taken over matrices which have $e^{-d_{\runij j}}\bm{1}_{r_{\runij}}$ as the 
$\beta_j$-th column in the $j$-th column block,
% ($j=1, \dots, m$), 
with the other columns being of the form 
$e^{-d_{\runij j}}(\bm{c}_{\runij j}^{\runab\beta})$. 

Now we define the polynomial $F_{X,\mathbf{r}} \in \mathbb{R}[a_{ij}^{\alpha\beta}]$. For each \(i,j\), let \(\bm{d}_{ij} = d_{ij} \bm{1}_{r_i}\). Then \(F_{X,\mathbf{r}}({a}_{i,j}^{\alpha\beta})\) is defined to be
\[
\sum_{\substack{\beta_1=1, \dots, r_1\\ \cdots\\ \beta_m=1, \dots, r_m}}
\det
\begin{pmatrix}
    \dots&
    e^{-d_{1j}}(\bm{a}_{1j}^{\runab 1} - \bm{d}_{1j})&
    \dots&
    \stackrel{\beta_j}{\widehat{\delta_{1j}(\bm{1}_{r_1})}}&
    \dots&
    e^{-d_{1j}}(\bm{a}_{1j}^{\runab r_j} - \bm{d}_{1j})&
    \dots\\
    \dots&
    e^{-d_{2j}}(\bm{a}_{2j}^{\runab 1} - \bm{d}_{2j})&
    \dots&
    \delta_{2j}(\bm{1}_{r_2})&
    \dots&
    e^{-d_{2j}}(\bm{a}_{2j}^{\runab r_j} - \bm{d}_{2j})&
    \dots\\
    &\vdots&\ddots&\vdots&\ddots&\vdots\\
    \dots&
    e^{-d_{mj}}(\bm{a}_{mj}^{\runab 1} - \bm{d}_{mj})&
    \dots&
    \delta_{mj}(\bm{1}_{r_m})&
    \dots&
    e^{-d_{mj}}(\bm{a}_{mj}^{\runab r_j} - \bm{d}_{mj})&
    \dots
\end{pmatrix}
\]
where $\delta_{ij}$ is the Kronecker delta.

Rewriting the formula (\ref{eq:coeff}), but now picking out the $\beta_1$-st, $\cdots$, $\beta_n$-th 
columns, we see that the coefficient of \((-t)^{n-m}\)  in the denominator of \eqref{eq:magXt} is equal to
\[
\sum_{\substack{\beta_1=1, \dots, r_1\\ \cdots\\ \beta_m=1, \dots, r_m}}
\det
\begin{pmatrix}
    \dots&
    \stackrel{\mbox{\tiny $\beta_1$-th of $1$-st block}}{\widehat{e^{-d_{11}}(\bm{1}_{r_1})}}&
    \dots&
    \stackrel{\mbox{\tiny $\beta_j$-th of $j$-th block}}{\widehat{e^{-d_{1j}}(\bm{1}_{r_1})}}&
    \dots&
    \stackrel{\mbox{\tiny $\beta_m$-th of $m$-th block}}{\widehat{e^{-d_{1m}}(\bm{1}_{r_1})}}&
    \dots\\
    \dots&
    e^{-d_{21}}(\bm{1}_{r_2})&
    \dots&
    e^{-d_{2j}}(\bm{1}_{r_2})&
    \dots&
    e^{-d_{2m}}(\bm{1}_{r_2})&
    \dots\\
    &\vdots&\ddots&\vdots&\ddots&\vdots\\
    \dots&
    e^{-d_{m1}}(\bm{1}_{r_m})&
    \dots&
    e^{-d_{mj}}(\bm{1}_{r_m})&
    \dots&
    e^{-d_{mm}}(\bm{1}_{r_m})&
    \dots
\end{pmatrix},
\]
and, applying Lemma \ref{lem:elem3}, we see that this is equal to
\(
F_{X,\mathbf{r}}(d_{ij}^{\alpha\beta}) \cdot \det(e^{-d_{ij}})
\).
By assumption, $\det(e^{-d_{ij}}) \neq 0$. Thus, if \(F_{X,\mathbf{r}}(d_{ij}^{\alpha\beta}) \neq 0\), then the first non-zero term in the denominator of \eqref{eq:magXt} is
\[
F_{X,\mathbf{r}}(d_{ij}^{\alpha\beta}) \cdot \det(e^{-d_{ij}}) (-t)^{n-m}.
\]

Next, we consider the numerator of \eqref{eq:magXt}. 
Note that $\sumofentry (\adj (Z_{X_t}))$ is equal to 
\[
\sum_{\substack{j=1, \dots, m\\ \beta=1, \dots, r_j}}
\det
\begin{pmatrix}
    \dots&
    e^{-d_{1j}^{\runab 1}(t)}&
    \dots&
    \stackrel{\beta}{\widehat{\bm{1}_{r_1}}}&
    \dots&
    e^{-d_{1j}^{\runab r_j}(t)}&
    \dots\\
    \dots&
    e^{-d_{2j}^{\runab 1}(t)}&
    \dots&
    \bm{1}_{r_2}&
    \dots&
    e^{-d_{2j}^{\runab r_j}(t)}&
    \dots\\
    &\vdots&\ddots&\vdots&\ddots&\vdots\\
    \dots&
    e^{-d_{mj}^{\runab 1}(t)}&
    \dots&
    \bm{1}_{r_m}&
    \dots&
    e^{-d_{mj}^{\runab r_j}(t)}&
    \dots\\
\end{pmatrix}. 
\]
We have an expansion using (\ref{eq:exponential}), as in the case of the denominator. 
Unlike in that case, each determinant could be non-zero at degree 
$n-m-1$, where the coefficient is
\[
\sum_{\substack{j=1, \dots, n\\ \beta=1, \dots, r_j}}
\sum_{\substack{\beta_1=1, \dots, r_1\\ \cdots\\ \beta_n=1, \dots, r_n}}
\det
\begin{pmatrix}
    \dots&
    e^{-d_{1j}}(\bm{c}_{1j}^{\runab 1})&
    \dots&
    \stackrel{\beta_j}{\widehat{e^{-d_{1j}}(\bm{1}_{r_1})}}&
    \dots&\stackrel{\beta}{\widehat{\bm{1}_{r_1}}}&\dots&
    e^{-d_{1j}}(\bm{c}_{1j}^{\runab r_j})&
    \dots\\
    \dots&
    e^{-d_{2j}}(\bm{c}_{2j}^{\runab 1})&
    \dots&
    e^{-d_{2j}}(\bm{1}_{r_2})&
    \dots&\bm{1}_{r_2}&\dots&
    e^{-d_{2j}}(\bm{c}_{2j}^{\runab r_j})&
    \dots\\
    &\vdots&\ddots&\vdots&\ddots&\vdots&\ddots&\vdots\\
    \dots&
    e^{-d_{nj}}(\bm{c}_{nj}^{\runab 1})&
    \dots&
    e^{-d_{nj}}(\bm{1}_{r_n})&
    \dots&\bm{1}_{r_n}&\dots&
    e^{-d_{nj}}(\bm{c}_{nj}^{\runab r_j})&
    \dots
\end{pmatrix}. 
\]
However, by the alternating property of the determinant (using exchange between the 
$\beta_j$-th and $\beta$-th columns), this becomes zero. Thus, the lowest degree 
is again $n-m$, and that term has coefficient
\[
\sum_{j=1}^m
\sum_{\substack{\beta_1=1, \dots, r_1\\ \cdots\\ \beta_m=1, \dots, r_m}}
\det
\begin{pmatrix}
    \dots&
    \stackrel{\mbox{\tiny $\beta_1$-th of $1$-st block}}{\widehat{e^{-d_{11}}(\bm{1}_{r_1})}}&
    \dots&
    \stackrel{\mbox{\tiny $\beta_j$-th of $j$-th block}}{\widehat{\bm{1}_{r_1}}}&
    \dots&
    \stackrel{\mbox{\tiny $\beta_m$-th of $m$-th block}}{\widehat{e^{-d_{1m}}(\bm{1}_{r_1})}}&
    \dots\\
    \dots&
    e^{-d_{21}}(\bm{1}_{r_2})&
    \dots&
    \bm{1}_{r_2}&
    \dots&
    e^{-d_{2m}}(\bm{1}_{r_2})&
    \dots\\
    &\vdots&\ddots&\vdots&\ddots&\vdots\\
    \dots&
    e^{-d_{m1}}(\bm{1}_{r_m})&
    \dots&
    \bm{1}_{r_m}&
    \dots&
    e^{-d_{mm}}(\bm{1}_{r_m})&
    \dots
\end{pmatrix}. 
\]
Again applying \Cref{lem:elem3}, this turns to be equal to 
\[
F_{X,\mathbf{r}}(d_{ij}^{\alpha\beta}) \cdot
\sum_{j=1}^m\det
\begin{pmatrix}
    e^{-d_{11}} & \dots & \stackrel{\text{$j$-th column}}{\widehat{1}} & \dots & e^{-d_{1m}}\\
    e^{-d_{21}} & \dots & 1 & \dots & e^{-d_{1m}}\\
    \vdots & \ddots & \vdots & \ddots & \vdots\\
    e^{-d_{m1}} & \dots & 1 & \dots & e^{-d_{mm}}
\end{pmatrix}, 
\]
which is precisely 
$F_{X,\mathbf{r}}(d_{ij}^{\alpha\beta}) \cdot\sumofentry(\adj(Z_X))$. 

We have shown that, for \(0 < t \ll 1\), the magnitude of \(Y_{f,t}\) can be expressed as 
\[
|Y_{f,t}|=\frac{F_{X,\mathbf{r}}(d_{ij}^{\alpha\beta}) \cdot\sumofentry(\adj(Z_X))\cdot (-t)^{n-m}+O(t^{n-m+1})}{F_{X,\mathbf{r}}(d_{ij}^{\alpha\beta}) \cdot\det(Z_X)\cdot (-t)^{n-m}+O(t^{n-m+1})}.
\]
It follows that if $F_{X,\mathbf{r}}(d_{ij}^{\alpha\beta}) \neq 0$, then \(\lim_{t\to 0}|Y_{f,t}|=|X|\).
\end{proof}

It only remains to prove that the polynomial \(F_{X,\mathbf{r}}\) is not constantly zero.

\begin{proposition}
\label{prop:poly_nonzero}
Let \(X\) be any \(m\)-point metric space. For each \(n \geq m\) and \(\mathbf{r} = (r_1,\ldots,r_m)\) such that \(\sum r_i = n\), there exists $f:Y\surj X$ with $\# f^{-1}(x_i)=r_i$, such that $F_{X,\mathbf{r}}(d_Y(y_i^\alpha, y_j^\beta)) \neq 0$.
\end{proposition}

\begin{proof}
Let \(Y\) be the metric space with underlying set \(\bigsqcup_{i=1}^m \{y_i^1, \ldots, y_i^{r_i}\}\) and
\[
d_Y(y_i^\alpha, y_j^\beta)=
\begin{cases}
d_X(x_i, x_j) & i\neq j\\
1 & i=j \text{ and }\alpha\neq\beta\\
0 & i=j \text{ and }\alpha=\beta.
\end{cases}
\]
Let \(f\colon Y \surj X\) be given by \(f(y_i^\alpha) = x_i\).

Using the fact that $c_{ij}^{\alpha\beta}=0$ for $i\neq j$, and $c_{ii}^{\alpha\beta}=1-\delta_{\alpha\beta}$, 
we have 
\[
\begin{split}
F_{X,\mathbf{r}}(d_{ij}^{\alpha\beta})
&=
\sum_{\substack{\beta_1=1, \dots, r_1\\ \cdots\\ \beta_m=1, \dots, r_m}}
\prod_{k=1}^m\det
\begin{pmatrix}
c_{kk}^{11} & c_{kk}^{12} & \dots & \stackrel{\beta_k}{\widehat{1}}&\dots&c_{kk}^{1 r_k}\\
c_{kk}^{21} & c_{kk}^{22} & \dots & 1 &\dots&c_{kk}^{2 r_k}\\
\vdots & \vdots & \ddots & \vdots&\ddots&\vdots\\
c_{kk}^{r_k 1} & c_{kk}^{r_k 2} & \dots & 1 &\dots&c_{kk}^{r_k r_k}
\end{pmatrix}\\
&=
\sum_{\substack{\beta_1=1, \dots, r_1\\ \cdots\\ \beta_m=1, \dots, r_m}}
\prod_{k=1}^n\det
\begin{pmatrix}
0 & 1 & \dots & \stackrel{\beta_k}{\widehat{1}}&\dots & 1 \\
1 & 0 & \dots & 1 &\dots& 1 \\
\vdots & \vdots & \ddots & \vdots&\ddots&\vdots\\
1 & 1 & \dots & 1 &\dots&0
\end{pmatrix}\\
&=(-1)^{n-m} r_1r_2\cdots r_m.
\end{split}
\]
Since \(f\) is a surjection, we have \(r_i \neq 0\) for all \(i\), so this product is non-zero.
\end{proof}

\bibliographystyle{amsplain}
\bibliography{mag_generic_continuity_v2.bib}

\end{document}